\documentclass[]{amsart}

\allowdisplaybreaks[1]

\usepackage{amssymb,amsfonts,amsmath,amscd,latexsym,amsxtra}
\usepackage{euler, charter}
\theoremstyle{plain}
\newtheorem{theorem}{Theorem}[section]
\newtheorem{corollary}[theorem]{Corollary}

\newtheorem{lemma}[theorem]{Lemma}
\newtheorem{proposition}[theorem]{Proposition}
\newtheorem*{theoremquote}{Theorem}
\newtheorem*{propositionquote}{Proposition}

\theoremstyle{example}

\theoremstyle{definition}
\newtheorem{definition}[theorem]{Definition}
\newtheorem*{ack}{Acknowledgement}

\theoremstyle{remark}
\newtheorem{remark}[theorem]{Remark}
\theoremstyle{notation}

\theoremstyle{e}

\numberwithin{equation}{section}

\newcommand{\Spec}{\operatorname{Spec}}
\newcommand{\Hom}{\operatorname{Hom}}

\newcommand{\End}{\operatorname{End}}

\newcommand{\GL}{\operatorname{GL}}

\newcommand{\Mqn}{\mathcal{M}_{q,n}}
\newcommand{\Aqn}{A_{q,n}}
\newcommand{\Nqn}{\mathcal{N}_{q,n}}

\newcommand{\Iqn}{\mathcal{I}_{q,n}}
\newcommand{\UI}{\mathcal U}
\def\bA{\mathbb A}

\def\cO{\mathcal O}
\def\bP{\mathbb P}
\def\bC{\mathbb C}
\DeclareMathOperator{\Hilb}{Hilb}
\DeclareMathOperator{\pHilb}{\Hilb^n_{[0]}}

\def\a{\alpha}

\input xy
\xyoption{all}

\input epsf
\DeclareMathOperator{\stab}{Stab}
\def\GAqnk{\operatorname{GA}_{q,n}}
\def\gl{\mathfrak{gl}}
\def\Aut{{\mathrm Aut}}

\begin{document}
\title{Conjugacy classes of  commuting  nilpotents}
\author[W. J. Haboush]{William J. Haboush}
\address{Department of Mathematics \\ 273 Altgeld Hall \\ 1409 W. Green St.\\ The University of Illinois at Urbana Champaign \\ Urbana, Illinois \\ 61801}
\email{haboush@math.uiuc.edu}
\author[D. Hyeon]{Donghoon Hyeon}
\address{Department of Mathematical Sciences \\ Seoul National University \\ Seoul, 151-747\\ R. O. Korea}
\email{dhyeon@snu.ac.kr}

%\subjclass{}
%\keywords{}

\begin{abstract}
We consider the space $\Mqn$ of regular $q$-tuples of  commuting nilpotent endomorphisms of $k^n$ modulo simultaneous conjugation. We show that $\Mqn$ admits a natural homogeneous space structure, and that it is an affine space bundle over $\bP^{q-1}$. A closer look at the homogeneous structure reveals that, over $\bC$ and with respect to the complex topology, $\Mqn$ is a smooth vector bundle over $\bP^{q-1}$. We prove that, in this case, $\Mqn$ is diffeomorphic to a direct sum of twisted tangent bundles. We also prove that $\Mqn$ possesses a universal property and represents a functor of ideals, and use it to identify $\Mqn$ with an open subscheme of a punctual Hilbert scheme. By using a result of \cite{Iarrobino}, we show that $\Mqn \to \bP^{q-1}$ is not a vector bundle, hence giving a family of affine space bundles that are not vector bundles.

\end{abstract}
\maketitle
%\tableofcontents

\section{Introduction}
 The space of commuting matrices $ \{(A, B) \in \gl_n\oplus\gl_n \, | \, [A B] = 0\}$,
 oft called the {\it commuting variety},  had received a fair amount of attention especially in regard to the irreducibility question \cite{MT1, MT2, G, GS}. The {\it nilpotent commuting variety} is the closed subvariety consisting of commuting nilpotent pairs which has also been researched extensively by several authors: Baranovsky \cite{Bar} ($\textup{char}(k) = 0$ or $\textup{char}(k) > n$) and Basili
\cite{Bas} ($\textup{char}(k) = 0$ or $\textup{char}(k) \ge n/2$) showed that the space  is irreducible and has dimension $n^2-1$.
Baranovsky went on to conjecture that the corresponding variety for any complex semisimple
Lie algebra $\mathfrak g$ is equidimensional of dimension dim $\mathfrak g$, which was confirmed positively
by Premet \cite{P}.

In this article, we shall be concerned with the space $\Mqn$ of regular $q$-tuples of commuting nilpotents $N_1, \dots, N_q$ $\in$ $\End(V)$ up to simultaneous conjugation, where $k$ is an algebraically closed field and $V$ is a $k$-vector space of dimension $n$. A $q$-tuple $(N_1,\dots,N_q)$ is said to be {\it regular} if $N_i^{n-1} \ne 0$ for some $i$.  It is rather surprising that this natural space, with a definite moduli theory flavor,  has completely evaded any research activity. Indeed, it turns out that $\Mqn$ exhibits a rich and interesting geometry which will be revealed throughout the paper.  In the rest of the introduction, we shall give a detailed overview of our results and a roadmap for the paper.

We begin by associating to a $q$-tuple $(N_1, \dots, N_q)$ of commuting nilpotents in $\End(V)$, a $k$-algebra homomorphism
\[
\begin{array}{cccc}
\rho: & k[x_1, \dots, x_q] & \to & \End(V) \\
& x_i & \to & N_i.
\end{array}
\]
That is, we have an associated {\it representation} of $A_{q,n} = k[x_1,\dots,x_q]/\langle x_1,\dots, x_q\rangle^n$. In Section~\ref{S:Aqn}, we shall prove that two cyclic representations $\rho, \rho'$ are isomorphic if and only if $\ker(\rho) = \ker(\rho')$ (Proposition~\ref{P:equiv}). Moreover, an annihilator $\ker(\rho)$ of a regular representation is of the form $\a(\mathfrak q_1)$ for some automorphism $\a$ of $A_{q,n}$, where $\mathfrak q_1 = \langle x_2, x_3, \dots, x_q \rangle$ (Lemma~\ref{tran}). Hence the space $\Mqn$ of regular $q$-tuples of commuting nilpotents modulo conjugation can be realized as the orbit space $\Aut(A_{q,n})/G_1$ where $G_1$ is the stabilizer of $\mathfrak q_1$. The subsequent sections are devoted to the study of the  structure of this orbit space.

In Section~\ref{S:algebraic_groups}, we define relevant algebraic groups and gather some basic properties which will be employed in the study of $\Mqn$.
 The parabolic group $P_1 = \GL_q(k) \cap G_1$, the group $\mathcal I_{q,n} = \ker(\Aut(A_{q,n}) \to \GL_q(k))$ of linearly trivial automorphisms and the quotient group $\mathcal I_{q,n}/(\mathcal I_{q,n}\cap G_1)$ play especially important roles. These are affine group schemes and we compute their dimensions.

In Section~\ref{S:quasihs}, we  show that $\Mqn$ as a homogeneous space is isomorphic to an  affine space bundle over $\bP^{q-1}$ with fibre isomorphic to the quotient group $\mathcal I_{q,n}/(\mathcal I_{q,n}\cap G_1)$ (Proposition~\ref{fsp}):

\begin{propositionquote}   We have an isomorphism $
  \mathcal{M}_{q,n} \simeq \GL_q(k)\times ^{P_1}\mathcal{I}_{q,n}/(\mathcal{I}_{q,n}\cap G_1).
 $
where $P_1$ acts on $\mathcal{I}_{q,n}/(\mathcal{I}_{q,n}\cap G_1)$ by conjugation. In particular, $\Mqn$ is an equivariant  bundle of relative dimension $(q-1)(n-2)$ over $\mathbb{P}^{q-1}_k$.
\end{propositionquote}

Due to this proposition, it is clear that the $P_1$-space structure of the quotient group $\mathcal{I}_{q,n}/(\mathcal{I}_{q,n}\cap G_1)$  is the key for understanding the bundle structure of $\Mqn$ over $\bP^{q-1}$.  For this end,  we further investigate the structure of groups $\mathcal I_{q,n}$ and $\mathcal I_{q,n}/(\mathcal I_{q,n}\cap G_1)$ in Section~\ref{S:Iqn}. Corollary~\ref{C:G-mod} gives a co-ordinate system under which $\mathcal I_{q,n}/(\mathcal I_{q,n}\cap G_1)$ is identified with an affine space, and this  allows us to study the affine bundle structure of $\Mqn$ in an explicit manner in the subsequent section.

In Section~\ref{S:topology}, we build on the results from Section~\ref{S:Iqn} and study the topology of $\Mqn$ as a complex manifold over $\bC$. Algebraically, the $P_1$-space structure of $\mathcal I_{q,n}/(\mathcal I_{q,n}\cap G_1)$ is not very well behaved. But once we pass to the smooth category, the $P_1$-structure is easy to understand due to our work in Section~\ref{S:Iqn}. We prove that

\begin{theoremquote} The moduli space $\mathcal M_{q,n}$ as a smooth fibre bundle is isomorphic to the direct sum $\oplus_{j=2}^{n-1} T_{\bP^{q-1}}(+j)$ of twisted tangent bundles.
\end{theoremquote}

In Section~\ref{S:universal}, we shall prove that $\Mqn$ is a fine moduli scheme in the sense of algebro-geometric moduli theory.  The space $\Mqn$ is an orbit space $\Aut(A_{q,n})/G_1$ parametrizing length $n$ ideals of $A_{q,n}$, and we can show that the induced sheaf $\mathcal I(q_1)$ on it has a universal property. It follows from this that

\begin{theoremquote}
The space $\Mqn$ is a fine moduli scheme for the moduli functor $\underline{\Mqn} : \mathrm{Sch}/k \to \mathrm{Sets}$ from the category of $k$-schemes to the category of sets, defined by
\[
\underline{\Mqn}(S) = \{ \mbox{Ideal sheaves $\mathcal I \subset \cO_S \otimes_k A_{q,n}$ ARR and flat over $S$} \}
\]
where ARR is short for ``annihilates regular representations" (Definition~\ref{D:ARR}).
\end{theoremquote}
By using the universal property, we can readily identify $\Mqn$ with an open subscheme of a suitable Hilbert scheme. Our main theorem below will be proved in  Section~\ref{S:hilbert}.
\begin{theoremquote}  $\mathcal M_{q,n}$ is isomorphic to an open subscheme of the punctual Hilbert scheme $\Hilb^n_{[0]}\bP^{q}$.
\end{theoremquote}

Finally, we use a theorem by Iarrobino \cite{Iarrobino} and show that
\begin{theoremquote} $\Mqn \to \bP^{q-1}$ is not a vector bundle in the algebraic category. 
\end{theoremquote}
This is quite interesting on its own: Examples of affine space fibre bundles that are not vector bundles are quite rare. 

 \begin{ack} Hanspeter Kraft had reviewed an early version of this article and gave us many helpful suggestions
which led us to
 Proposition~4.2 and Proposition~6.2.  He deserves to be a co-author but graciously declined to be listed, and we would like to sincerely thank him for his contributions. 

We are also grateful to Michel Brion and Young-Hoon Kiem helpful discussion and suggestions. 

 D.H was partially supported by the Research Resettlement Fund for the new faculty of Seoul
National University, the  SNU Invitation Program for Distinguished Scholar, and
 the following grants funded by the government of Korea:
NRF grant 2011-0030044 (SRC-GAIA) and NRF-2013R1A1A2010649.
\end{ack}

\section{The Artinian algebra $\Aqn$ and its representations} \label{S:Aqn}

To a $q$-tuple $(N_1, \dots, N_q)$ of commuting nilpotents in $\End(V)$, one can associate a $k$-algebra homomorphism
\[
\begin{array}{cccc}
\phi: & k[x_1, \dots, x_q] & \to & \End(V) \\
& x_i & \to & N_i.
\end{array}
\]
Since $N_i$'s are nilpotent, $\ker(\phi)$ contains the ideal $J_{q,n}$ generated by all forms of degree $n$ in the variables $x_1, \dots ,x_q.$ Of course, $J_{q,n} = \mathfrak m_0^n$ where  $\mathfrak m_0$ is the maximal ideal generated by $x_1,\dots, x_q$.

%\noindent We shall use the multi-vector notation $f(x) = f(x_1, \dots, x_q)$.

\begin{definition} Let $\Aqn =k[x_1,\dots ,x_q]/J_{q,n}$ and $\mathfrak{m} = \mathfrak{m}_0/J_{q,n}.$  We shall call $\Aqn$ {\it the ring of $n$-nil polynomials}.
\end{definition}
Clearly $\Aqn$ is an Artinian $k$-algebra of $k$-dimension $\binom{n+q-1}{n-1}.$

\

A representation of $\Aqn$ will mean a $k$-algebra homomorphism $\rho: \Aqn \to \End(V)$.
Through a representation $\rho$, $V$ is endowed with an $\Aqn$-module structure. We denote it by $V_\rho$.  The correspondence between $q$-tuples of commuting $n$-nilpotents in $\End(V)$ and representations of $\Aqn$ is bijective. For this reason, we consider
\[
\Nqn := \left\{  \left. (N_1,\dots, N_q) \in \prod^q\bA^{n^2} \, \right|  \,  N_i^n = 0, [N_i, N_j] = 0, \forall i, j   \right\}
\]
as the variety of representations of $\Aqn$ in $V$, regarded as a subvariety of the $q$-fold product of the affine $n^2$-space with underlying vector space $\End(V)$.

\begin{definition} \label{D:regular_rep} A representation $\rho$ is called {\it regular} if $\rho(u^{n-1}) \ne 0$ for some $u\in \mathfrak m$. It is said to be {\it cyclic} if there is a vector $v\in V$ such that $\rho(\Aqn).v=V.$
\end{definition}

\begin{definition} \label{D:regular}A $q$-tuple $(N_1, \dots, N_q)$ of commuting nilpotents in $\End(V)$ is said to be {\it regular} (resp. {\it cyclic}) if the corresponding representation $\rho : A_{q,n} \to \End(V)$ determined by $\rho(x_i) = N_i$ is regular (resp. cyclic).
\end{definition}

Evidently a regular representation is cyclic. Also, the regular (resp. cyclic) representations form a Zariski open subset $\mathcal N^{r}_{q,n}$ (resp. $\mathcal N^c_{q,n}$) of $\Nqn$ since it is the complement of the vanishing locus of a collection of suitable minors. We write $\Mqn$ (resp. $\mathcal M^c_{q,n}$) for the set of equivalence classes of points of $\mathcal N^r_{q,n}$ (resp. $\mathcal N^c_{q,n}$) under simultaneous conjugation. Clearly $\Mqn$ is a proper subset of $\mathcal M^c_{q,n}$. Note that we are dropping the superscript $r$ when we pass from $\mathcal N^r_{q,n}$ to the quotient $\Mqn$ for notational convenience later on.

\

Given a representation $\rho$, we write $\mathcal{A}(\rho)$ for the annihilator in $\Aqn$ of $V_{\rho}$.
\begin{proposition}\label{P:equiv}  Let $\rho$ and $\rho'$ be two cyclic representations of $\Aqn$. Then $V_{\rho}$ and $V_{\rho'}$ are isomorphic as $\Aqn$-modules if and only if $\mathcal{A}(\rho)= \mathcal{A}(\rho').$ The isomorphism classes of cyclic representations of $\Aqn$ are in bijective correspondence with ideals in $\Aqn$ of codimension $n$ and with the simultaneous conjugacy classes of ordered $q$-tuples of commuting nilpotents.
\end{proposition}

\begin{proof} If $\rho$ is cyclic with cyclic generator $v$ then the annihilator of $v$ is equal to the annihilator of $V_{\rho}$ and the map $a\mapsto av$ induces an isomorphism between $\Aqn/\mathcal{A}(\rho)$ and $V_{\rho}.$ It is trivial that if two representations are isomorphic then their annihilators are equal. All that must be shown is that if $I$ is an ideal in $\Aqn$ of codimension $n$ then there is a cyclic representation $(\rho, V)$ with annihilator $I=\mathcal{A}(\rho).$

If $I$ is of codimension $n$ then $\Aqn/I$ is a vector space of dimension $n$ and so there is an isomorphism (of $k$-vector spaces) $\theta:\Aqn/I \to V.$ Define $\rho$ by the equation $\rho(a)v=\theta(a\theta^{-1}(v)).$ This gives a representation of $\Aqn$ with cyclic vector $\theta(1).$
\end{proof}

\begin{lemma} Let $(Z_1,\dots ,Z_q)$ be a regular $q$-tuple i.e.  $Z_i^{n-1}\neq 0$ for some $i$. Then
\begin{enumerate}
\item any $Z_j$ is a polynomial in $Z_i$, and;
 \item if $(\rho, V)$ is the corresponding  representation, then  there is a $k$-algebra isomorphism $\rho(\Aqn)\simeq k[z]/z^nk[z].$
 \end{enumerate}
\end{lemma}

\begin{proof} Suppose that $Z_1$ is of rank $n-1$, and let $v$ be a vector not annihilated by $Z_1^{n-1}$. Then $\{v, Z_1v, Z_1^2v, \dots, Z_1^{n-1}v\}$ form a basis of $V$: if $\sum_{i=0}^{n-1}a_iZ^iv = 0$ with $a_0 = a_1 = \dots = a_{m-1}= 0$ and $a_m \ne 0$, then we would have
\[
Z^{n-m-1}(a_mZ^m v) = - Z^{n-m-1}(a_{m+1}Z^{m+1}v + \dots a_{n-1}Z^{n-1}v) = 0
\]
which is contradictory to our choice of $v$. With respect to this ordered cyclic basis, $Z_1$ is represented by the $n\times n$ matrix $X$ with $1$ on the sub diagonal and zero elsewhere. An elementary computation shows that any $n\times n$ matrix commuting with $X$ is a polynomial in $X$. The first item follows, and we let $f_i$ denote the polynomial without a constant term such that $f_i(Z_1) = Z_i$. Then the kernel of $\rho$ is generated by $x_i - f_i(x_1)$, so $\rho(A_{q,n})$ is isomorphic to $A_{q,n}/\langle x_2-f_2(x_1), \dots, x_q-f_q(x_1)\rangle \simeq k[x_1]/\langle x_1^n \rangle$.
\end{proof}

\begin{definition} We let $\GAqnk$ denote the algebraic group of $k$-algebra automorphisms of $A_{q,n}$.
\end{definition}
Note that $\GAqnk$ is a linear algebraic group since it is a closed subgroup of $\GL(A_{q,n})$ where $A_{q,n}$ is viewed as a finite dimensional $k$-vector space. 

\begin{definition} A subset $\{ u_1, \dots ,u_r \} \subset \Aqn$ will be called a {\it system of nil parameters} of $\Aqn$ of length $r$ if the $\Aqn/\langle u_1,\dots,u_r\rangle$ is isomorphic  to $A_{q-r,n}.$  The variables $x_1, \dots, x_q$ are called the {\it standard nil parameters}.
\end{definition}

\begin{remark} Throughout this paper,  we shall let \textit{ $\mathfrak q_1$ denote the ideal of $A_{q,n}$ generated by $x_2, \dots, x_q$.}
\end{remark}

\begin{lemma}\label{tran}Let $\rho$ be a regular representation. Then there is a system of nil parameters, $u_1, \dots, u_q$ so that $A(\rho)$ is the ideal generated by $u_2, \dots, u_q.$ Moreover there is an automorphism $\alpha \in \GAqnk$ such that $\alpha(x_i)=u_i$ and $\alpha(\mathfrak{q}_1)=A(\rho).$

\end{lemma}

\begin{proof} Since $\rho$ is regular,  $\rho(u_1)^{n-1}\neq 0$ for some $u_1 \in \mathfrak m$. We have seen that this implies that the image of $\rho$ is of the form $k[z]/z^nk[z]$ where $\rho(u_1)$ is  the class of $z.$ Let $\mathfrak{q}=\operatorname{ker}(\rho).$ Now $\rho$ induces the map $\varphi: \mathfrak m/\mathfrak m^2 \to (\bar z)/(\bar z^2)$ of cotangent spaces. Choose $v_2,\dots ,v_q \in \mathfrak m$  whose images in $\mathfrak m/\mathfrak m^2$ form a basis for the kernel of $\varphi$. Then there is a polynomial $f_j(z)\in z^2k[z]$ such that $\rho(v_j)\equiv f_j(z)$ modulo $z^nk[z].$ Hence $v_j-f_j(u_1)\in \mathfrak{q}$ and it is congruent to $v_j$ modulo $\mathfrak{m}^2.$ Let $u_j=v_j-f_j(u_1)$, $j=2,\dots, q.$ The elements $u_1, u_2, \dots ,u_q$ are a basis for $\mathfrak{m}/\mathfrak{m}^2$, so they generate $\Aqn$ as a $k$-algebra. That is, they are a set of nil parameters for $\Aqn.$ The elements $u_2,\dots ,u_q$ generate an ideal $\mathfrak{q}'\subseteq \mathfrak{q}$ and $\Aqn/\mathfrak{q}'\simeq k[z]/z^nk[z].$ Hence by dimension count $\mathfrak{q}=\mathfrak{q}'.$ It immediately follows that we may define an automorphism $\alpha \in \GAqnk$ such that $\alpha(x_i)=u_i$ and $\alpha(\mathfrak{q}_1)=\mathfrak{q}=A(\rho).$
\end{proof}

\section{Preliminary results on the algebraic groups of automorphisms of $A_{q,n}$}\label{S:algebraic_groups}

We let $\Omega$ denote the $k$-vector subspace of $\Aqn$ generated by $x_1, \dots, x_q$. By abusing notation, we shall identify $\Omega$ with the cotangent space $\mathfrak{m}/\mathfrak{m}^2$.
Let $\sigma$ be an automorphism of $\Aqn.$ Then $\sigma(\mathfrak{m})=\mathfrak{m}$ and so $\sigma$ induces a linear automorphism of $\Omega$.  The map which assigns to $\sigma$ the associated automorphism of $\Omega \simeq \coprod_{i=1}^qk\bar x_i$ is clearly a group morphism from the group $\GAqnk$ of automorphisms of $\Aqn$ to $\GL(\Omega).$

\begin{definition} Let $\pi : \GAqnk \to \GL(\Omega)$ denote the morphism which sends the automorphism $\sigma$ of $\Aqn$ to the associated linear automorphism in $\GL(\Omega)$. Also, we identify $\GL(\Omega)$ with $\GL_q(k)$ by using the basis $\bar x_1, \dots, \bar x_q$.
\end{definition}

Clearly $\pi$ is surjective: Let  $\Omega_1$ be the subspace of $\Omega$ spanned by $x_2, \dots, x_q.$  Any $\alpha = (\alpha_{ij})\in \GL_q(k)$ naturally defines an element $\tilde\alpha \in \GAqnk$ defined by $\tilde{\alpha}(x_i) = \sum_j \alpha_{ij}x_j$. This defines a section $\chi : \GL_q(k) \to \GAqnk$ and we frequently identify $\GL_q(k)$ with its image in $\GAqnk$ as the group of linear automorphisms of $\Aqn$.
An automorphism $\sigma \in \GAqnk$ will be called {\it linearly trivial} if it is in  the kernel of $\pi.$ If $\sigma$ is linearly trivial, there are quadratic expressions $s_i(x_1,\dots ,x_q)\in \mathfrak{m}^2,\; i=1,\dots ,q$ such that $\sigma(x_i)=x_i-s_i.$ We let $\mathcal{I}_{q,n}$ denote the kernel of $\pi$:

\begin{equation}\label{fues}
1\rightarrow \mathcal{I}_{q,n}\rightarrow \GAqnk\rightarrow \GL_q(k)\rightarrow 1.
\end{equation}

We also wish to describe the stabilizer $G_1$ of the ideal $\mathfrak{q}_1$.
Note that $\pi$ carries $G_1$ to the stabilizer $P_1$ in $\GL_q(k)$ of the codimension one space $\Omega_1$ in $\Omega=\mathfrak{m}/\mathfrak{m}^2$. Then the exact sequence (\ref{fues}) restricts to the exact sequence:

\begin{equation}\label{stes}
1\rightarrow \mathcal{I}_{q,n}\cap G_1 \rightarrow G_1\rightarrow P_1\rightarrow 1.
\end{equation}

The section $\chi$ carries $P_1$ to the stabilizer of $\Omega_1$ in $\GL_q(k).$ In consequence $\GAqnk$ and $G_1$ are compatibly semidirect products
$\GL_q(k)\cdot \mathcal{I}_{q,n}$ and $P_1\cdot(\mathcal{I}_{q,n}\cap G_1)$ respectively so that the action of an element of $\GL_q(k)$ on $\Aqn$ is determined by its linear action on $\Omega.$

\

We first offer some preliminary results on the algebraic groups involved.

\begin{proposition}\label{vstr}The exact sequences (\ref{fues}) and (\ref{stes}) with the section $\chi$ induce isomorphisms of varieties:
\begin{align*}
\GAqnk&\simeq \GL_q(k)\times \mathcal{I}_{q,n} \\
 G_1\simeq &P_1\times (\mathcal{I}_{q,n}\cap G_1)
\end{align*}
Moreover:
\begin{enumerate}
  \item The groups $\mathcal{I}_{q,n}$ and $\mathcal{I}_{q,n}\cap G_1$  are unipotent. Let $\mathcal{I}^j_{q,n}$ denote the subgroup of $\mathcal{I}_{q,n}$ consisting of elements which reduce to the identity modulo $\mathfrak{m}^{j+2}.$ Each of these groups is normal in $\GAqnk$
  \item There is  an isomorphism of additive group schemes $\mathcal{I}^j_{q,n}/\mathcal{I}^{j+1}_{q,n}\simeq (\mathfrak{m}^{j+2}/\mathfrak{m}^{j+3})^q$.
  \item There are natural isomorphisms of varieties:
  \begin{align}
     \mathcal{I}_{q,n}\simeq &(\mathfrak{m}^2)^{(q)}\label{iso1} \\ \mathcal{I}_{q,n}\cap G_1 &\simeq \mathfrak{m}^2\times (\mathfrak{q}_1\mathfrak{m})^{(q-1)}\label{iso2}\end{align} The exponents in parentheses represent iterated Cartesian products.
     \item The groups $\GAqnk$ and $G_1$ are connected affine group schemes.
\end{enumerate}
\end{proposition}

\begin{proof} The two initial product decompositions follow from the existence of the section $\chi.$ We define a map $\alpha$ from $\mathcal{I}_{q,n}$ to $(\mathfrak{m}^2)^{(q)}$ as follows:  For $\sigma \in \mathcal{I}_{q,n}$, $ \sigma(x_i)=x_i+u_i$ for some unique element $u_i\in \mathfrak{m}^2.$ Let $\alpha(\sigma)=(u_1,\dots ,u_q)$, the corresponding $q$-tuple of elements of $\mathfrak{m}^2.$ Conversely given any such $q$-tuple consider the elements $x'_i=x_i+u_i.$ Let $A=k[x'_1, \dots x'_q]$ the subalgebra of $\Aqn$ generated by these elements. The $x'_i$ generate $\mathfrak{m}/\mathfrak{m}^2$ and so the associated graded of $A$ and that of $\Aqn$ are the same. Hence the two algebras are of the same dimension and the $x'_i$ constitute a system of nil parameters. Thus there is an automorphism of $\Aqn$ sending $x_i$ to $x'_i$ and so $\alpha$ is surjective as well as injective. This establishes (\ref{iso1}).

The element $\sigma \in \mathcal{I}_{q,n}$ is in $\mathcal{I}_{q,n}\cap G_1$ if and only if  $\sigma(x_i)\in \mathfrak{q}_1$ for $i\geq 2.$ Now $\sigma(x_i)=x_i+u_i$ with $u_i\in \mathfrak{m}^2$ and so $\sigma \in \mathcal{I}_{q,n}\cap G_1$ if and only if $u_i\in \mathfrak{q}_1\cap \mathfrak{m}^2$ for each $i\geq 2.$ Said otherwise, $\sigma \in \mathcal{I}_{q,n}\cap G_1$ if and only if $\alpha(\sigma)\in \mathfrak{m}^2\times (\mathfrak{q}_1\mathfrak{m})^{(q-1)}.$ Thus (\ref{iso2}) is established.

 For the first item, if $\mathcal{I}_{q,n}$ is unipotent then its subgroup $\mathcal{I}_{q,n}\cap G_1$ is as well. To see that $\mathcal{I}_{q,n}$ is unipotent we examine the filtration of item 1. The surjection $\Aqn \rightarrow A_{q,r},\;r<n$ induces a surjective map $\GAqnk\rightarrow GA_{q,r}$ with kernel $\mathcal{I}_{q,n}^{r-2}$ whence each of the groups is normal in $\GAqnk.$ Hence if $\sigma\in \mathcal{I}_{q,n}^j$ it induces the identity on $\Aqn/\mathfrak{m}^{j+2}.$ Consequently $\sigma(x_i)=x_i+u_i(x_1,\dots,x_q),\; u_i\in \mathfrak{m}^{j+2}.$ Suppose that $\tau$ is another such automorphism and that $\tau(x_i)=x_i+v_i(x_1,\dots ,x_q)$ with $v_i\in \mathfrak{m}^{j+2}.$ Then $\sigma \circ \tau(x_i)=\sigma(x_i+v_i(x_1,\dots ,x_q))=x_i+u_i(x_1,\dots ,x_q)+v_i(x_1+u_1, \dots ,x_q+u_q).$ Now $v_i$ is a polynomial with lowest degree terms of degree at least $j+2$ and so we may write $v_i(x_1+u_i, \dots ,x_q+u_q)=v_i(x_1,\dots ,x_q)+\sum_{(\nu_1,\dots ,\nu_q)}c_{(\nu_1, \dots ,\nu_q)}u_1^{\nu_1}\dots u_q^{\nu_q}.$ In this last expression the $q$-tuples $(\nu_1, \dots ,\nu_q)$ have nonnegative integral entries with at least one positive entry and the coefficients $c_{(\nu_1, \dots ,\nu_q)}$ are polynomials in the $x_i$ with no constant terms. Consequently, since $u_i$ is of degree at least $j+2$ all terms except the first are in $\mathfrak{m}^{j+3}.$ In particular, $\sigma\circ \tau(x_i)\equiv x_i+ u_i(x_1,\dots ,x_q)+v_i(x_1,\dots ,x_q)$ \, $\text{mod} (\mathfrak{m}^{j+3}).$ For $\sigma \in \mathcal{I}_{q,n}^j$ let $\alpha_j(\sigma)=(\overline{\sigma(x_1)-x_1}, \dots , \overline{\sigma(x_q)-x_q})$ where the overline denotes class in $\mathfrak{m}^{j+2}/\mathfrak{m}^{j+3}.$ The preceding calculation shows that $\alpha_j$ is a homomorphism to the additive group scheme $(\mathfrak{m}^{j+2}/\mathfrak{m}^{j+3})^q.$ For any $q$-tuple $(u_1,\dots ,u_q)$ with all the $u_i\in \mathfrak{m}^{j+2}$ the elements $x_i+u_i$ are a system of nil parameters and so there is an automorphism $\sigma$ sending $x_i$ to $x_i+u_i$ for each $i.$ It follows that $\alpha_j$ is surjective. Hence $\mathcal{I}_{q,n}$ admits a finite filtration so that successive quotients are group schemes of additive type isomorphic to the vector spaces $\mathfrak{m}^{j+2}/\mathfrak{m}^{j+3}.$ Hence it is unipotent and connected.
\end{proof}

We compute the dimensions of the groups $\mathcal{I}_{q,n}$ and $G_1\cap \mathcal{I}_{q,n}$ and relevant homogeneous spaces.

\begin{proposition}\label{dcomp}
\begin{enumerate}
\item $\operatorname{dim}(\mathcal{I}_{q_n})=q(\binom{q+n-1}{n-1}-(q+1))$
\item $\operatorname{dim}(\mathcal{I}_{q,n}\cap G_1)= \binom{q+n-1}{n-1}+(q-1)(\binom{q+n-1}{n-1}-(q+1)-(n-2))$
\item $\operatorname{dim}(\GAqnk)=q\binom{q+n-1}{n-1}-q$
\item $\operatorname{dim}(G_1)=q\binom{q+n-1}{n-1}-qn+n-1$
\item $\operatorname{dim}(\GAqnk/G_1)=(q-1)(n-1)$
\item $\operatorname{dim}(\mathcal{I}_{q,n}/(\mathcal{I}_{q,n}\cap G_1))=(q-1)(n-2)$
\end{enumerate}

\end{proposition}

\begin{proof}First note that $\Aqn$ is graded and isomorphic as a vector space to the sum of the first $n-1$ symmetric powers of the vector space of dimension $q.$ Hence its dimension is equal to the dimension of the forms of degree $n-1$ in $q+1$ variables, that is to say $\binom{q+n-1}{n-1}.$ Hence the dimension of $\mathfrak{m}$ is $\binom{q+n-1}{n-1}-1$ and the dimension of $\mathfrak{m}^2$ is $\binom{q+n-1}{n-1}-q-1.$

Now we wish to compute the dimensions of the ideals $\mathfrak{q}_1$ and $\mathfrak{q}_1\mathfrak{m}.$ Now $\Aqn/\mathfrak{q}_1\simeq k[\bar z]$ where $\bar z$ is the residue class of $z$ in $k[z]/z^nk[z].$ The isomorphism sends the residue class of $x_1$ to $\bar z.$ Hence $\operatorname{dim}(\mathfrak{q}_1)=\binom{q+n-1}{n-1}-n.$

Now note that $\mathfrak{q}_1\cap \mathfrak{m}^2 =\mathfrak{q}_1\mathfrak{m}.$ To see this note that an element of $\mathfrak{m}^2$ can be written in the form $x_1^2f(x_1)+u$ with $u\in \mathfrak{q}_1\mathfrak{m}.$ Thus modulo $\mathfrak{q}_1$ this just becomes $(\bar z)^2f(\bar z)$ and so it lies in $\mathfrak{q}_1$ if and only if $x_1^2f(x_1)=0.$ This means that $\mathfrak{q}_1\cap \mathfrak{m}^2 =\mathfrak{q}_1\mathfrak{m}.$ Hence we may compute the dimension of $\mathfrak{q}_1\mathfrak{m}$ from the exact sequence:
$$ 0\rightarrow \mathfrak{q}_1\mathfrak{m}\rightarrow \mathfrak{m}^2\rightarrow (\bar z)^2k[\bar z]\rightarrow 0 $$
In consequence $\operatorname{dim}(\mathfrak{q}_1\mathfrak{m})=\binom{q+n+1}{n-1}-(q+1)-(n-2)=\binom{q+n+1}{n-1}-q$  $-n+1.$

To compute the dimensions of the groups $\mathcal{I}_{q,n}$ and $\mathcal{I}_{q,n}\cap G_1$ we recall the isomorphisms (\ref{iso1}) and (\ref{iso2}). Hence as a variety, $\mathcal{I}_{q,n}$ is isomorphic to the $q$-fold Cartesian product $(\mathfrak{m}^2)^q.$ That is it is isomorphic to affine $k$-space of dimension $q(\binom{q+n-1}{n-1}-(q+1)).$

Applying (\ref{iso2}), $\mathcal{I}_{q,n}\cap G_1$ is isomorphic to the the product of the vector space $\mathfrak{m}^2$ and the $(q-1)$'th Cartesian power of the affine space $\mathfrak{q}_1\mathfrak{m}.$ We find that $\operatorname{dim}(\mathcal{I}_{q,n}\cap G_1)$ is
\[
\binom{q+n-1}{n-1}-(q+1)+(q-1)(\binom{q+n-1}{n-1}  -(q+1)-(n-2)).
\]
Since $\operatorname{dim}(\GL_q(k))=q^2$ and $\operatorname{dim}(P_1)=q^2-q+1$ we may use the exact sequences (\ref{fues}) and (\ref{stes}) to compute the dimensions of $\GAqnk$ and $G_1.$ The formulae in the theorem can now be obtained from some basic algebra computations combined with these results.  \end{proof}

\section{The quasihomogeneous structure on $\Mqn$} \label{S:quasihs}

Due to Lemma~\ref{tran}, we can identify the space $\Mqn$ with the orbit $\GAqnk.\mathfrak q_1$, giving it a homogeneous space structure. From now on,

\begin{definition}\label{D:Mqn} $\Mqn$ means the homogeneous space $\GAqnk.\mathfrak q_1 \simeq \GAqnk/G_1$.
\end{definition}

\begin{proposition}\label{fsp}  $\Mqn$ is an equivariant  bundle of relative dimension $(q-1)(n-2)$ over $\mathbb{P}^{q-1}_k$.  More precisely, we have an isomorphism
 \[
  \mathcal{M}_{q,n} \simeq \GL_q(k)\times ^{P_1}\mathcal{I}_{q,n}/(\mathcal{I}_{q,n}\cap G_1).
 \]
 Here, $P_1$ acts on $\mathcal{I}_{q,n}/(\mathcal{I}_{q,n}\cap G_1)$ by conjugation i.e. $p.[\sigma] = [p\circ \sigma \circ p^{-1}]$ where $\circ$ denotes the multiplication (composition) in the automorphism group $\GAqnk$.
\end{proposition}

\begin{proof} Since $\mathcal{I}_{q,n}$ is normal in $\GAqnk$, $P_1$ normalizes it and we may consider the semidirect product $H=P_1\cdot \mathcal{I}_{q,n}$.
Recall the exact sequence (\ref{stes}).  Since $P_1$ can be viewed as a subgroup of $\GAqnk$ as above, we may write the stabilizer $G_1$ as the semidirect product $P_1\cdot (\mathcal{I}_{q,n}\cap G_1).$ In particular $G_1\subseteq H.$ By Lemma~\ref{tran}, $\GAqnk$ operates transitively on $\mathcal{M}_{q,n}$ which can be written as the orbit $\GAqnk.\mathfrak{q}_1=\GAqnk/G_1.$ 

Since $G_1\subseteq H$ there is a natural $\GL_q(k)$-equivariant fibration 
\[
\varpi:\GAqnk/G_1\mapsto \GAqnk/H
\]
with the base  $\GAqnk/H=(\GAqnk/\mathcal{I}_{q,n})/(H/\mathcal{I}_{q,n})=\GL_q(k)/P_1=\mathbb{P}_k^{q-1}$. Since $\varpi$ is a $P_1$-equivariant fibration, it equals $GL_q(k) \times^{P_1} \varpi^{-1}(P_1)$. The fiber  $\varpi^{-1}(P_1)$ is $H/G_1$, where the groups on the top and the bottom are the semidirect products $P_1\cdot \mathcal{I}_{q,n}$ and $P_1 \cdot (\mathcal{I}_{q,n}\cap G_1)$, respectively. Hence the fibre  is isomorphic to $\mathcal{I}_{q,n}/\mathcal{I}_{q,n}\cap G_1$. It has dimension $(q-1)(n-2)$ by Lemma~\ref{dcomp}. Note that $P_1$ acts on  $P_1\cdot \mathcal{I}_{q,n}/P_1 \cdot (\mathcal{I}_{q,n}\cap G_1)$ by left multiplication and on $\mathcal{I}_{q,n}/\mathcal{I}_{q,n}\cap G_1$ by conjugation, and a $P_1$-space isomorphism from the former to the latter is given by $
[px] \mapsto [pxp^{-1}]$.

\end{proof}

In order to complete our examination of the structure of $\mathcal{M}_{q,n}$ we shall determine the structure of the fiber $\mathcal{I}_{q,n}/(\mathcal{I}_{q,n}\cap G_1)$ as a $P_1$-variety. Then by examining the $P_1$-space structure, over $k = \bC$, we shall see that $\mathcal{M}_{q,n}$ is diffeomorphically isomorphic to a sum of twisted tangent bundles over $\mathbb{P}^{q-1}.$

\section{More on the group $\mathcal{I}_{q,n}$ } \label{S:Iqn}

We wish to understand the structure of $\mathcal{M}_{q,n}$ as a fibre bundle over $\mathbb{P}_k^{q-1}.$ By Proposition~\ref{fsp}, it is a homogeneous fiber space whose fiber is the $P_1$-space $\mathcal{I}_{q,n}/(\mathcal{I}_{q,n}\cap G_1)$. Note that  $\mathcal{I}_{q,n}/(\mathcal{I}_{q,n}\cap G_1)$ has a $P_1$-structure through conjugation since $P_1$ normalizes both groups. To analyze $\Mqn$, it is crucial to understand the $P_1$-space structure of $\mathcal{I}_{q,n}/(\mathcal{I}_{q,n}\cap G_1)$.

We begin with a product decomposition for $\mathcal{I}_{q,n}$ and an invariant version of (\ref{iso1}). First we consider a group $\Gamma \subseteq \mathcal{I}_{q,n}$ which we will demonstrate to be a geometric complement to $\mathcal{I}_{q,n}\cap G_1.$

\begin{definition}\label{sect} $$\Gamma=\{ \sigma \in \GAqnk  : \sigma(x_1)=x_1, \: \sigma(x_i)=x_i+x_1^2f_i(x_1),\; i\geq 2\}$$
\end{definition}

\begin{proposition}\label{P:Gamma}The multiplication morphism $\mu: \Gamma \times (\mathcal{I}_{q,n}\cap G_1)\rightarrow \mathcal{I}_{q,n}$ is an isomorphism of varieties.\end{proposition}

\begin{proof}Since $\Gamma\cap (\mathcal{I}_{q,n}\cap G_1)=\{ e\}$, it is clear that the  $\mu$ is an injective morphism of varieties. Since each $f_i$ is has degree $2 \le \deg(f_i) \le n-1$, $\dim \Gamma = (q-1)(n-2)$. Hence by Proposition~\ref{dcomp}, the dimension of $\Gamma \times (\mathcal{I}_{q,n}\cap G_1)$ is equal to the dimension of $\mathcal{I}_{q,n}.$ We may regard $\mathcal{I}_{q,n}$ as a $\Gamma\times (\mathcal{I}_{q,n}\cap G_1)$-space with the action $(u,v)\cdot g=ugv^{-1}.$ Then $\operatorname{Im}(\mu)$ is the $\Gamma\times (\mathcal{I}_{q,n}\cap G_1)$-orbit of $e.$ Hence it is open in its closure which is $\mathcal{I}_{q,n}$. Since it is a product of affine spaces, it is affine and hence its complement is, if nonempty, of codimension one given as the zero set of a polynomial $f$ on $\mathcal{I}_{q,n}$. But $f|\operatorname{Im}(\mu)$ is a non-constant unit, which by a well known theorem of Rosenlicht is a character. This is impossible since $\mathcal{I}_{q,n}$ is unipotent and has only  trivial characters. 
\end{proof}

Now consider the ring $k[z]$ where $z^n=0.$ Any automorphism $\phi$ of $k[z]$ is uniquely determined by $\phi(z).$ The element $\phi(z)$ can be any element in $zk[z]$ not in $z^2k[z]$.  Refer to such elements as {\it generating elements}. If $u$ is a generating element then $u^{n-1}\neq 0$ and $u^n=0.$ Given any two generating elements $u_1$ and $u_2$ there is a unique automorphism $\phi$ such that $\phi(u_1)=u_2.$ As before, we let $\Omega_1$ denote the subspace of $A_{q,n}$ generated by $x_2, \dots, x_q$.

\begin{proposition}\label{P:lt-orbit} Let $\mathfrak{q}'$ be an ideal in $\Aqn$ generated by a system of nil parameters such that $\mathfrak{q}' = \Omega_1$ modulo $\mathfrak{m}^2$. Then there is a unique homomorphism $\phi:\Aqn \mapsto k[z]$ so that $\phi(x_1)=z$ and $\operatorname{Ker}(\phi)=\mathfrak{q}'.$ Moreover there
is a uniquely determined linear map $\gamma:\Omega_1\rightarrow x_1^2k[x_1] \subset A_{q,n}$ so that the elements $u-\gamma(u)$, $\forall u \in \Omega_1$, generate $\mathfrak{q}'.$ In particular the elements $x_i-\gamma(x_i),\;i>1$,  are a set of nil parameters of length $q-1$ generating $\mathfrak{q}'.$
\end{proposition}

 \begin{proof} Since $\mathfrak{q}'$ is generated by a set of nil parameters of length $q-1$ there is an isomorphism $\beta:\Aqn/\mathfrak{q}'\to k[z].$ Let $\phi_0$ be the composition of $\beta$ with the natural surjection $\Aqn \to \Aqn/\mathfrak{q}'.$ Since $\phi_0$ is surjective and the image of its kernel $\mathfrak{q}'$ in $\mathfrak{m}/\mathfrak{m}^2$ is equal to the image there of $\Omega_1$, $\phi_0(x_1)$ must be a generating element. Hence there is a unique automorphism $\theta$ of $k[z]$ carrying $\phi_0(x_1)$ to $z.$ Replacing $\phi_0$ by $\theta \circ \phi_0=\phi$ we may assume that  $\phi(x_1)=z.$

Since $\mathfrak{q}'=\Omega_1$ modulo $\mathfrak{m}^2,$ it follows that for each $u\in \Omega_1$ there is an element $u-s_u$ with $s_u\in \mathfrak{m}^2$ so that $\phi_0(u-s_u)=0.$  Now $\phi$ carries $\mathfrak{m}^2$ to $z^2k[z]$ and so $\phi(u)=\phi(s_u)$ is a polynomial in $z$ with neither a constant nor a linear term. Let $\phi(u)=f_u(z).$ Define a map $\gamma:\Omega_1\mapsto x_1^2k[x_1] \subset \Aqn$ by setting $\gamma(u)$ equal to $f_u(x_1).$ It is clear that $\gamma$ is a linear map since it is a composition of the linear map $\phi|\Omega_1$ and the inverse of the isomorphism $k[x_1]\mapsto k[z].$ Moreover  $\phi(u-\gamma(u))=0$ for all $u;$ that is these elements are in $\mathfrak{q}'.$ Consider the elements $x_i-\gamma(x_i)$, $i>1.$ These elements clearly constitute a system of nil parameters of length $q-1.$ Hence they generate an ideal of codimension $n$ which is contained in $\mathfrak{q}'$. It follows that they generate $\mathfrak{q}'.$

Since $\mathfrak{q}'$ is generated by a system of nil parameters of length $q-1$, namely $x_i'=x_i-\gamma(x_i)$, we have $\Aqn = k[x_1, x_2',\dots x_q']$ and the uniqueness of $\phi$ follows since $\phi$ is determined by $\phi(x_1)=z$ and $\phi(x_i')=0$, $i = 2, 3, \dots, q$.
\end{proof}

Note that there is an obvious isomorphism between $\Gamma$ and $\Hom_k(\Omega_1, z^2k[z])$, and we shall identify the two in the following corollary.

\begin{corollary}\label{C:G-mod} There is a natural isomorphism of varieties $\psi:\mathcal{I}_{q,n}/(\mathcal{I}_{q,n}\cap G_1)\rightarrow \Hom_k(\Omega_1, z^2k[z]).$\end{corollary}

\begin{proof} Consider the morphism 
\[
\tilde\psi: \mathcal I_{q,n} \stackrel{\mu^{-1}}{\longrightarrow} \Gamma\times (\mathcal I_{q,n} \cap G_1) \stackrel{\pi_1}{\longrightarrow} \Gamma
\]
where $\mu$ is the multiplication morphism (which is an isomorphism due to Proposition~\ref{P:Gamma}) and $\pi_1$ is the projection onto the first factor. The subgroup $\mathcal I_{q,n}\cap G_1$ acts on $\mathcal I_{q,n}$ by the right multiplication and on the product $\Gamma\times (\mathcal I_{q,n} \cap G_1)$ by the right multiplication on the second factor. Obviously $\tilde\psi$ is $\mathcal I_{q,n}\cap G_1$-invariant, so it descends to give the desired morphism $\psi$ on $\mathcal I_{q,n}/\mathcal I_{q,n}\cap G_1$. The ideals $\mathfrak q'$ in Proposition~\ref{P:lt-orbit} are precisely the ideals in the $\mathcal I_{q,n}$-orbit of $\mathfrak q_1$, and hence the statement of the proposition means that $\psi$ is bijective. Now, it follows that $\psi$ is an isomorphism since it is a bijective morphism between affine spaces.
\end{proof}

In order to give a complete description of $ \mathcal{M}_{q,n}=\GL_q(k)\times ^{P_1}\mathcal{I}_{q,n}/(\mathcal{I}_{q,n}\cap G_1)$, we must yet provide an explicit formula for the action of $P_1$ on $ \mathcal{I}_{q,n}/(\mathcal{I}_{q,n}\cap G_1).$ For this end, we need to concretely describe the $P_1$-action on $\Hom_k(\Omega_1, z^2k[z])$ that makes $\psi$ in Corollary~\ref{C:G-mod} $P_1$-equivariant. This will be taken up in the subsequent section.

\section{$\Mqn$ as a smooth  fibre bundle} \label{S:topology} \def\bq{\mathfrak q}
In this section, we shall work over $k = \bC$ and work in the category of smooth manifolds: We regard the algebraic groups $GA_{q,n}$, $\mathcal I_{q,n}$ etc. as (complex) Lie groups. We will also consider homogenous spaces $\mathcal I_{q,n}/(\mathcal I_{q,n}\cap G_1)$ and $GA_{q,n}/G_1$ as smooth manifolds, and the isomorphism 
\[
\psi: \mathcal I_{q,n}/(\mathcal I_{q,n}\cap G_1) \to \Hom(\Omega_1,z^2k[z]) \simeq \bC^{(q-1)(n-2)}
\]
and other relevant maps as smooth maps.

Let $f \in \mathcal I_{q,n}$ and $s_i = \gamma(x_i)$ be the polynomials (in $z$) associated to $f(\bq_1)$  as in the Proposition~\ref{P:lt-orbit}. 
There is a natural commutative diagram
\[
\xymatrix{
  \mathcal{I}_{q,n} \ar[dr]^-{\eta} \ar[d] &  \\
      \mathcal{I}_{q,n}/(\mathcal{I}_{q,n}\cap G_1)\ar[r]_-{\psi} & \Hom(\Omega_1,z^2k[z])
}
\]
where $\eta(f)$ maps $x_i$ to $s_i(z)$. Then the $P_1$ action on $\Hom_k(\Omega_1,z^2k[z])$ can be written in terms of $s_i$ as follows. Let $p=(p_{ij}) \in P_1$ and $f(\bq_1) \in \mathcal I_{q,n}\bq_1$.
 Then by Proposition~\ref{P:lt-orbit}, we have $f(\bq_1) = (x_2 - s_2(x_1), \dots, x_q - s_q(x_1))$   for some polynomials  $s_i(z) = \sum_{j=2}^{n-1} b_{ij}z^j$ with no constant and linear term, and
 \[
 p.f(\bq_1) = p.(x_2-s_2(x_1), \dots, x_q - s_q(x_1)) =  \left(\sum_j p_{ij}x_j - s_i\left(\sum_j p_{1j}x_j\right)\right)_{i\ge 2}.
 \]
Setting $g = p^{-1}$ and performing a Gauss elimination, we get
\[
\begin{array}{lll}
(\sum_j p_{ij}x_j - s_i(\sum_j p_{1j}x_j))_{i\ge 2} &=&(\sum_ig_{ki}(\sum_jp_{ij}x_j - s_i(\sum_j p_{1j}x_j)))_{i,k\ge 2} \\
&=& (x_k - \sum_ig_{ki}s_i(\sum_j p_{1j}x_j))_{k\ge 2}.
\end{array}
\]
This is a fairly complicated action and is not well understood in general. But in the special case where $p_{1j} = 0$ for $j\ge 2$, the action is linear. Then in terms of $(b_{ij})$, $p.f(\bq_1)$ corresponds to
$$((p^{-1})_{ik}b_{kj}p_{11}^j).$$
Restating this in terms of representations, we obtain the following.

\begin{lemma} Let  $t \in \bC$ and $\varphi_t : P_1 \to P_1$ be the group homomorphism
\[
p=\left[
\begin{array}{rrr}
p_{11} & R \\
0 & Q
\end{array}
\right]
\mapsto
\left[
\begin{array}{rrr}
p_{11} & tR \\
0 & Q
\end{array}
\right]
\]
where $R$, $0$, $Q$ have sizes $1\times (q-1)$, $(q-1)\times 1$ and $(q-1)\times (q-1)$ respectively.  Let $Y_t$ be the $P_1$-space whose underlying variety is $\mathcal I_{q,n}/\mathcal I_{q,n}\cap G_1$ on which $P_1$ acts by
\[
p.[\sigma] = [\varphi_t(p)\sigma \varphi_t(p)^{-1}].
\]
Then $Y_0$ is isomorphic to  $\coprod_{j=2}^{n-1}\Omega_1^*\otimes \bC_j$ where $\bC_j$ is $\bC$ acted upon by the $j$-th power of the character $\lambda(p) = p_{11}$, $\forall p \in P_1$.
\end{lemma}
Note that for $t\ne 0$, $Y_t \simeq Y_1$ as $P_1$-spaces  since $\varphi_t(p) = \tau p \tau^{-1}$, $\tau =\left[
\begin{array}{ll}
 t & 0\\
0 & I_2
\end{array}
\right]$ so that $\Phi_t([\sigma]) =  [\tau \sigma \tau^{-1}]$ is a $P_1$-equivariant isomorphism from $Y_1$ to $Y_t$. 
Let $B = \Spec k[t]$ and let $\mathcal Y = Y\times B$ be the $P_1$-space on which $P_1$ acts by
$p.([\sigma], t) = ([\varphi_t(p)\sigma \varphi_t(p)^{-1}], t)$.
 Consider $\mathcal E := \GL_q \times^{P_1} \mathcal Y \to \GL_q/P_1 \times B \simeq \bP^{q-1}\times B$. It is an isotrivial family of affine $(q-1)(n-2)$-bundles over $\bP^{q-1}$: We have
\[
E_t:=\mathcal E|(\bP^{q-1}\times \{t\}) = \GL_q\times^{P_1}Y_t \simeq \GL_q\times^{P_1}Y_1 \simeq \mathcal M_{q,n}.
\]

Recall that two smooth manifolds $X_1, X_2$ are said to be {\it deformation equivalent} if there exists a smooth family $\mathcal X \to S$ over a smooth connected base $S$ and two points $s_1, s_2 \in S$ such that the fibres $\mathcal X_{s_i}$ are diffeomorphic to $X_i$, $i = 1, 2$.

\begin{proposition}\label{P:def-equiv} The moduli space $\mathcal M_{q,n}$ as a smooth fibre bundle over $\bP^{q-1}$ is smooth deformation equivalent to the direct sum
\[
\oplus_{j=2}^{n-1} T_{\bP^{q-1}}(+j)
\]
of twisted tangent bundles over $\bP^{q-1}$.
\end{proposition}
\begin{proof}
Consider the one-parameter family $E_t := \GL_q \times^{P_1} Y_t$ of affine bundles over $\bP^{q-1}$. Then $E_0 = \oplus_{j=2}^{n-1} T_{\bP^{q-1}}(+j)$ since $\Omega_1^*$ corresponds to the tangent bundle and $\bC_j$, to $\mathcal O_{\bP^{q-1}}(+j)$. Since $E_1 = \mathcal M_{q,n} $, it follows that
 as smooth fibre bundles $\mathcal M_{q,n}$  and $\oplus_{j=2}^{n-1} T_{\bP^{q-1}}(+j)$ are deformation equivalent.
 \end{proof}

\def\aut{\mathrm{Aut}}
 \def\bR{\mathbb R}
 
 Our moduli space
$\Mqn$ is a fibre bundle over $\GL_q(k)/P_1 = \bP^{q-1}$ with fibres $F: = \Iqn/\Iqn\cap G_1$. Since $F$ is an affine space, the theory of microbundles can be applied to show that $\Mqn$ as a smooth fibre bundle is isomorphic to a smooth vector bundle i.e. its structure group reduces to the general linear group:

\begin{theorem}\label{T:main} The moduli space $\mathcal M_{q,n}$ as a smooth fibre bundle is isomorphic to the direct sum $
\oplus_{j=2}^{n-1} T_{\bP^{q-1}}(+j)$ of twisted tangent bundles.
\end{theorem}

\begin{proof} First, note that there is a distinguished ``zero''-section $0_n : \bP^{q-1} \to \mathcal M_{q,n}$  defined by sending $\bar g \in \GL_q/P_1$ to $g.q_1$ or equivalently, to 
\[
(g, 0) \in \GL_q\times^{P_1}\Hom(\Omega_1,z^2\bC[z]/\langle z^n \rangle).
\]
Let $N$ denote the $\bR$-dimension of the affine space $\Hom(\Omega_1,z^2\bC[z]/\langle z^n \rangle)$. 
The isomorphism classes of smooth fibre bundles with fibres diffeomorphic to $\bR^N$ is functorially in bijective correspondence with the homotopy classes of maps from $\bP^{q-1}$ to $BDiff(\bR^N, 0)]$ where $BDiff(\bR^N,0)$ is the classifying space of the group $Diff(\bR^N,0)$ of diffeomorphisms of $\bR^N$ fixing $0$. By the {\it Alexander trick}, the natural inclusion 
\[
\GL_N(\bR) \hookrightarrow Diff(\bR^N, 0)
\]
 is a homotopy equivalence, and it follows that $(\Mqn \to \bP^{q-1}, 0_n)$ has a reduction of structure group to $\GL_N(\bR)$. More specifically, up to diffeomorphism, $ \Mqn \to \bP^{q-1}$ is isomorphic to the normal bundle of the zero section $0_n$. 
By Proposition~\ref{P:def-equiv}, the maps corresponding to the two smooth vector bundles $\Mqn$ and $\oplus_{j=2}^{n-1} T_{\bP^{q-1}}(+j)$ are homotopy equivalent. Hence they are isomorphic as smooth vector bundles over $\bP^{q-1}$. 

\end{proof}

 \section{Universal property of $\mathcal M_{q,n}$} \label{S:universal}
%\subsection{Universal ideal sheaf over $\mathcal M_{q,n}$}
Let $G := \GAqnk$. Recall $\mathcal M_{q,n} = G/G_1$ (by definition), $G_1 = \stab_G(q_1)$. We have observed in Proposition~\ref{fsp} that $\Mqn \simeq \GL_q\times^{P_1}\mathcal{I}_{q,n}/(\mathcal{I}_{q,n}\cap G_1)$.
From the natural sequence
\[
0 \to q_1 \to A_{q,n} \to A_{q,n}/q_1 \to 0
\]
of $G_1$-representations, we obtain the sequence of induced sheaves on $G/G_1$:
\[
0 \to \mathcal I_{G/G_1}(q_1) \to \mathcal I_{G/G_1}(A_{q,n}) \to \mathcal I_{G/G_1}(A_{q,n}/q_1) \to 0.
\]
\begin{definition}\label{D:ARR} An ideal sheaf $\mathcal I \subset A_{q,n}\otimes_k \cO_S$ over a $k$-scheme $S$ is said to {\it annihilate regular representations over $S$} (we abbreviate this as ``ARR over $S$") if there exists a map $\rho : A_{q,n}\otimes_k \cO_S \to V \otimes_k \cO_S$ of $\cO_S$-modules such that for all closed points $b \in S$, $\rho_b := \rho\otimes 1_{k(b)} : A_{q,n} \otimes_k k(b) \to V\otimes_k k(b)$ is a regular representation and  $\ker(\rho_b) =  \mathcal I\otimes k(b) $.
\end{definition}

We shall prove that $\mathcal I_{G/G_1}(q_1) $ is an ideal sheaf of  $\mathcal I_{G/G_1}(A_{q,n}) \simeq  \mathcal O_{\Mqn}\otimes_k A_{q,n} = \mathcal O_{\Mqn}[x_1,\dots,x_q]/{\mathfrak m_0^n}$ such that
\begin{enumerate}
\item[(i)] it   annihilates regular representations over $\Mqn$ ;
\item[(ii)] it is {\it universal} : For any $k$-scheme $S$ and an ideal $\mathcal I$ of $ \mathcal O_S\otimes A_{q,n}$ ARR over $S$,  there exists a unique morphism $f_{\mathcal I}: S \to \Mqn$ such that $(f_{\mathcal I}\times 1)^*\left(\mathcal I_{G/G_1}(q_1)\right) \simeq \mathcal I$, where $1$ is the identity morphism on $\Spec A_{q,n}$.
\end{enumerate}

\begin{remark} Recall from Proposition~\ref{P:equiv} and the definition preceding it that $\ker(\rho)$ in (i) above is called the annihilator of the  representation $\rho$ and is denoted by $\mathcal A(\rho)$.
\end{remark}

%\begin{remark} The ARR property of $\mathcal I$ over $S$ is equivalent to the regularity of the $q$-commuting nilpotents $(B_1, \dots, B_q)|s$ of $End(\mathcal V)|s$ for every closed point $s \in S$, where $\mathcal V = \mathcal O_S\otimes A_{q,n}$ and $B_i$ are defined by multiplication by $x_i$.
%\end{remark}

\begin{proposition}\label{P:ARR} The ideal sheaf $\mathcal I_{G/G_1}(q_1) $ is ARR over $\Mqn$.
\end{proposition}

\begin{proof}
By Corollary~\ref{C:G-mod} and Proposition~\ref{P:lt-orbit}, we may naturally identify the homogeneous space $\overline{\mathcal I}_{q,n} := \mathcal{I}_{q,n}/(\mathcal{I}_{q,n}\cap G_1)$ with $$\Hom_k(\Omega_1, z^2k[z]) = \Spec k[b_{ij} \, | \, {2\le i \le q, 2 \le j \le n-1}]$$ where the coordinates $b_{ij}$ are the obvious ones determined by
\[
h(x_i) = \sum_{j=2}^{n-1} b_{ij}(h)z^j, \quad \forall h \in \Hom_k(\Omega_1, z^2k[z]).
\]
 Hence on $\overline{\mathcal I}_{q,n}$, we have an ideal
\[
\UI' := \left\langle x_i - \sum_{j=2}^{n-1}b_{ij}x_1^j \right\rangle_{2 \le i \le q} \subset \frac{k[b_{ij}, x_1,\dots,x_q]}{\mathfrak m_0^n}.
\]
The right hand side is the global section ring $\Gamma(\cO_{\overline{\mathcal I}_{q,n}})\otimes_k A_{q,n}$. Subsequently, a versal ideal sheaf $\UI''$ on $\GL_q(k) \times \overline{\mathcal{I}}_{q,n}$ is obtained by further moving $\UI'$ around by the $\GL_q(k)$ action:
\begin{equation}\label{E:U''}
\UI'' := \left\langle \sum_{l=1}^q a_{il}x_l - \sum_{j=2}^{n-1}b_{ij}\left(\sum_{l=1}^q a_{1l}x_l\right)^j \right\rangle_{2 \le i \le q} \,  \subset \,  \frac{k[a_{l \, m}, b_{ij}, x_2,\dots,x_q]}{\mathfrak m_0^n}  \tag{$\dagger$}
\end{equation}
where $k[GL_q(k)] = k[a_{l \, m} \, | \, 1 \le l,m \le q]$ so that $k[a_{l \, m}, b_{ij}, x_2,\dots,x_q]/\mathfrak m_0^n$ is the global section ring of $\cO_{\GL_q(k) \times \overline{\mathcal{I}}_{q,n}}\otimes_k A_{q,n}$.

Since $P_1$ acts diagonally on $\GL_q(k) \times \overline{\mathcal{I}}_{q,n}$ by $p.(g, v) = (g p^{-1}, p. v)$, and $\UI''$ is obtained by letting $GL_q(k)$ act on $\UI'$, $\UI''$  is plainly invariant under the $P_1$ action and descends to give an ideal sheaf $\UI$ on the quotient $\Mqn = \left(\GL_q(k) \times \overline{\mathcal{I}}_{q,n}\right)/P_1$
\[
0 \to \UI \to \mathcal O_{\Mqn}\otimes_k A_{q,n} \to \mathcal W \to 0.
\]
By construction, $\UI$ is clearly ARR over $\Mqn$. It is easy to see that $\UI$ is isomorphic to $\mathcal I_{G/G_1}(q_1)$: Since they are ideal subsheaves of the same sheaf $\cO_{\Mqn} \otimes_k A_{q,n}$ of algebras, it suffices to show that their fibres are exactly the same in the fibres of $\cO_{\Mqn} \otimes_k A_{q,n}$. Let  $\sigma \in G$. The fibre of $\mathcal I_{G/G_1}(q_1)$ at $\sigma G_1$ is by definition $\sigma (q_1) $. By our analysis in the previous sections, especially Proposition~\ref{P:lt-orbit},
\begin{equation}\label{E:UI}
\sigma (q_1) = (a_{ij}).\left\langle \sum x_i + s_i(x_1)\right\rangle_{i=2,\dots,q}
\end{equation}
for a unique $((a_{ij}), (b_{ij})) \in \GL_q\times^{P_1} \overline{\mathcal I}_{q,n}$ where $b_{ij}$ are determined by $s_i(x_1) = \sum_{j=2}^{n-1} b_{ij}x_1^j$. Comparing (\ref{E:UI}) with Equation~(\ref{E:U''}), we see immediately that $\sigma(q_1)$ is precisely the fibre of $\mathcal U$ at $\sigma G_1$.

\end{proof}

\begin{proposition}\label{P:universal} The ideal sheaf $\mathcal I_{G/G_1}(q_1) $ is universal over $\Mqn$.
\end{proposition}

\begin{proof}
Let $S$ be a $k$-scheme and  $\mathcal I$ be an ideal of $ \mathcal O_S\otimes A_{q,n}$ ARR over $S$.
 Let $\mathcal V$ denote the quotient of $\mathcal O_S\otimes A_{q,n}$ by $\mathcal I$. Let $\pi : S\times \Spec A_{q,n} \to S$ be the projection and let $\mu_i  \in H^0(S, \pi_*End(\mathcal V))$ be the endomorphism of $\pi_*\mathcal V$ defined by multiplication by $x_i^{n-1}$. Since $\mathcal I$ is ARR, for each $s \in S$, there exists $i$ such that
\[
\mu_i|_s \ne 0 \in End(\pi_*\mathcal V)|_s.
\]
Hence there exists an open subvariety $T\ni s$ of  $S$ and a section $v \in \pi_*\mathcal V(T)$ such that $\mu_i|_t(v_t) \ne 0$ for all $t \in T$. At $s$, $\{v|_s, x_i.v|_s, x_i^2.v|_s, \dots, x_i^{n-1}.v|_s\}$ is a basis for $\mathcal V|_s \simeq k^n$. After shrinking $T$ if necessary, we may assume that $\{v, x_i.v, x_i^2.v, \dots, x_i^{n-1}.v\}$
is a framing of $\pi_*\mathcal V$ over $T$. That is, $x_i^\ell.v$'s give rise to an isomorphism
\[
\mathcal O_T^{\oplus n} \longrightarrow  \pi_*\mathcal V|_T
\]
over $T$. But with respect to this framing, multiplication by $x_i$ is represented by an $n\times n$ matrix $B$ with $1$'s on the subdiagonal and zeros elsewhere. Any matrix commuting with $B$ is easily shown to be a polynomial in $B$, and $\mu_j$'s commute with each other. So, we have over $T$,
\[
x_j = \sum_{l = 1}^{n-1} b_{jl}x_i^l \quad (\mbox{mod} \, \,  \mathcal I\otimes_{\cO_S}\cO_T), \quad b_{jl} \in \Gamma(\mathcal O_T), \quad j \ne i.
\]
By dimension reasons, $x_j - \sum_{l = 1}^{n-1} b_{jl}x_i^l$ generated $\mathcal I$ over $T$, and it induces a unique morphism $\psi_T : T \to \Mqn$ such that $(\psi_T\times 1)^*\UI = \mathcal I|_{T\times \Spec A_{q,n}}$. Indeed, we have a well-defined lifting $\tilde\psi_T: T \to \GL_n \times Hom_k(\Omega_1, z^2k[z])$ of $\psi_T$ given by
\[
t \mapsto (g, f_t)
\]
where $f$ is the homomorphism defined by $b_{jl}(t)$ and $g \in \GL_n$ is the permutation matrix $g.x_1 = x_i$, $g.x_i = x_1$ and $g.x_k = x_k$ for all $k \ne 1, i$. Then $f_t(\mathfrak q_1)$ is the ideal defined by $x_j - \sum_{l=1}^{n-1} b_{jl}(t)x_1^l$, $j \ne 1$, and $g.f_t(\mathfrak q_1)$ is precisely the ideal $\mathcal I|_{\{t\}\times \Spec A_{q,n}}$ which also is plainly seen to be equal to $\tilde\psi_T^*(\mathcal U'')$.

We claim that these $\psi_T$'s glue together to give the desired morphism $S \to \Mqn$. Indeed, consider $\psi_T$ and $\psi_{T'}$, given by $\tilde\psi_T(t) = (g, f_t)$ and $\tilde\psi_{T'}(t')=(g', f'_{t'})$. At any point $t \in T\cap T'$ in the intersection, we have $g.f_t(\mathfrak q_1) = \mathcal I_{\{t\}\times \Spec A_{q,n}} = g'.f'_t(\mathfrak q_1)$, and it follows that
\[
(g', f'_t) = (g p_t^{-1}, p_t. f_t)
\]
for some $T$-point $p$ of $P_1(T)$.
The dependence of $p_t$ on $t$ is certainly algebraic, and it follows that $\psi_T$ and $\psi_{T'}$ agree on $T\cap T'$, defining a morphism $T \cup T' \to \mathcal M_{q,n}$. Since $(\psi_T\times 1_{A_{q,n}})^*\mathcal U = \mathcal I_{T\times \Spec A_{q,n}}$ for each $T$, $\mathcal U$ is universal.
\end{proof}

\begin{theorem} \label{T:universal} The space $\Mqn$ is a fine moduli scheme for the moduli functor $\underline{\Mqn} : \mathrm{Sch}/k \to \mathrm{Sets}$ from the category of $k$-schemes to the category of sets, defined by
\[
\underline{\Mqn}(S) = \{ \mbox{Ideal sheaves $\mathcal I \subset \cO_S \otimes_k A_{q,n}$ ARR and flat over $S$} \}
\]
\end{theorem}
\begin{proof} The contents of Propositions~\ref{P:ARR} and \ref{P:universal} combined give the assertion of the theorem.
\end{proof}

\def\Nrqn{\mathcal N^r_{q,n}}

There is a natural set map $\pi: \Nrqn \to \Mqn$ that sends $(N_1, \dots, N_q)$ to the kernel of $\rho :  k[x_1, \dots, x_q] \to \End(k^n)$, $\rho(x_i) = N_i$.  From the universality of $\Mqn$, it follows immediately that this is a morphism of varieties:
\begin{corollary} There is a $\GL_n$ invariant orbit map $\pi: \mathcal N^r_{q,n} \to \Mqn$.
\end{corollary}

\begin{proof}  %Let $(N_1, \dots, N_q)$ be a point in $\Nrqn$. Let $a^{(s)}_{ij}$, $1\le s \le q$, $1 \le i,j \le n$ be a coordinate system for $k[\End(V)^{\oplus q}] = k[(\bA_k^{n^2})^q]$. By abusing notation, we may regard $N_s$ as the coordinates  $a^{s}_{ij}$, and for any $f[x_1, \dots, x_q] \in A_{q,n}$, $f(N_1, \dots, N_q)$
Consider the natural map
\[
\cO_{\mathcal N^r_{q,n}}\otimes_k A_{q,n} \to \cO_{\mathcal N^r_{q,n}}\otimes_k \End(V)
\]
defined by sending $f(x_1,\dots,x_q)\in \cO_{\mathcal N^r_{q,n}}(U)[x_1,\dots,x_q]$ to $f(N_1,\dots,N_q)$ where $(N_1,\dots,N_q)$ are by abuse of notation regarded as sections of $\cO_{\mathcal N^r_{q,n}}$ over a given open set $U \subset \Nrqn$. The kernel $\mathcal K$ of this map is an ideal sheaf ARR over $\Nrqn$ giving rise to a morphism $\Nrqn \to \Mqn$, which is constant on $\GL_n$ orbits due to Proposition~\ref{P:equiv}. By the universality of $\Mqn$ (Theorem~\ref{T:universal}), we conclude that $\pi$ is algebraic. 

\end{proof}

\section{The moduli space as an open subscheme of a punctual Hilbert scheme}\label{S:hilbert}

\def\bP{\mathbb P}

Let $X_0, \dots, X_q$ denote a set of homogeneous coordinates for $\bP^q$, $[0] := [1, 0, \dots, 0]$ and $x_i = X_i/X_0$ be the affine coordinates of the affine chart $\{X_0 \ne 0\}$. In this section, we shall consider the relation between $\Mqn$ and the Hilbert scheme $\Hilb^n\bP^q$ of length $n$ zero dimensional subschemes of $\bP^q$. The Hilbert scheme is defined by its functor. That is, for any scheme $S$ over $k$, we have
\[
\Hom(S, \Hilb^n\bP^q) = \{ Z \subset \bP^q_S \, | \, \mbox{$Z$ is surjective, finite flat over $S$ of degree $n$}\}.
\]
Let $\bP^{q(n)}$ be the $n$th symmetric product of $\bP^q$ that parametrizes length $n$ cycles of $\bP^q$. Recall the Hilbert-Chow morphism $\Psi : \Hilb^n\bP^q \to \bP^{q(n)}$ that maps a zero-dimensional subscheme $Z$ to its underlying cycle $[Z]$.

\begin{definition} The {\it punctual Hilbert scheme} $\Hilb^n_{[0]}\bP^q$ is the reduced fibre of the Hilbert-Chow morphism $\Psi$ over the cycle  $n\cdot [0]$.
\end{definition}
That is,  the punctual Hilbert scheme $Hilb^n_{[0]}\bP^{q}$ parametrizes the zero dimensional subschemes of length $n$ with support at the single point $[0] \in \bP^{q}$.
\def\cO{\mathcal O}

\begin{theorem}\label{T:hilb} $\mathcal M_{q,n}$ is isomorphic to an open subscheme of the punctual Hilbert scheme $\Hilb^n_{[0]}\bP^{q}$.
\end{theorem}
\begin{proof}
The universal quotient $\cO_{\mathcal M_{q,n}}\otimes A_{q,n} \to \mathcal W:=\cO_{\mathcal M_{q,n}}\otimes A_{q,n} /\mathcal I_{G/G_1}(\mathfrak q_1) $ gives rise to a map $\eta: \mathcal M_{q,n} \to \Hilb^n\bP^q$, by (abusing notation and) identifying $x_i$ with the affine coordinates $x_i$ of $
\bP^q$ at $[0]$. Since $x_i^n = 0$ for all $i$, $\mathcal W$ is supported at $[0]$. Also,  $\Mqn$ is nonsingular (hence reduced in particular) since it is an affine space bundle over $\bP^{q-1}$. It follows that the map $\eta$ factors through the punctual Hilbert scheme $\pHilb\bP^q$.

Consider the universal quotient sequence
\[
0 \to \mathcal J  \to \cO_{\Hilb_{[0]}^n\bP^q}[x_1,\dots,x_q] \to \mathcal Q \to 0
\]
on the punctual Hilbert scheme. Multiplication by $x_i. : \mathcal Q \to \mathcal Q$ defines an $\cO_{\Hilb_{[0]}^n\bP^q}$-linear map and the locus where $x_i^{n-1}$ is identically zero is a closed subscheme, say $Z_i$, of the punctual Hilbert scheme. Then $U := \pHilb\bP^q \setminus  \bigcap_i Z_i$ is an open subscheme of $\pHilb\bP^q$.
Now, the restriction of $\mathcal J$ to $U$ is ARR, by construction of $U$.  Hence by Proposition~\ref{P:universal}, we have the corresponding morphism $f_{\mathcal I}: U \to \mathcal M_{q,n}$, providing an inverse to $\eta$.

\end{proof}

%\begin{remark} It seems that whether the punctual Hilbert scheme $\pHilb(\bP^q)$ is irreducible is an open question. But for $q=2$, the question is settled (positively) for all $n$ over any field $k = \bar k$. We expect that $\mathcal M_{q,n}$ is the unique maximal irreducible component.\end{remark}

\def\gr{\mathrm {gr}} 
\section{$\Mqn \to \bP^{q-1}$ is not a vector bundle}
 Note that the action of $P$ on $\mathcal I_{q,n}/\mathcal I_{q,n}\cap G_1$ not being linear does not imply that $\Mqn = \GL_q \times^{P_1}  (\mathcal I_{q,n}/\mathcal I_{q,n}\cap G_1) \to \bP^{q-1}$ is not a vector bundle since the transition function may well become linear after passing to the quotient. Examples of algebraic  affine space fibre bundles that are not algebraic vector bundles are quite rare. We shall prove in this section that, as opposed to the result of the previous section, $\Mqn$ is not a vector bundle over $\bP^{q-1}$ by using the following theorem of Iarrobino:
 
 \begin{theorem}\label{T:Iarrobino}\cite[Theorem~1]{Iarrobino}  The variety parametrizing linear ideals $I \subset k[[x,y]]$ of co-length 4 is locally trivial over $\bP^1$ but is not a vector bundle.
 \end{theorem}
Here, an ideal $I \subset \langle x,y \rangle$ is said to be {\it linear} if $I $ is not contained in $\langle x,y \rangle^2$. It is equivalent to $I$ being an annihilator of a regular representation (Definition~\ref{D:ARR})  since    $f \in I$ satisfies $f^{n-1} \ne 0$ in $k[[x,y]]/\langle x,y \rangle^n$ if and only if $f \not\in \langle x,y\rangle^2 $

The map from $Z$ to $\bP^1$ is given by sending $I$ to its associated graded ideal $\gr(I)$ which is linear and also of co-length 4. 
We note that $Z \to \bP^2$ is precisely our fibration $\mathcal M_{2,4} \to \mathcal M_{2,2}$ induced by the natural map $A_{2,4} \to A_{2,2}$. The key ingredient of the proof of the above theorem is that there is no section of $\mathcal M_{2,4} \to \mathcal M_{2,3} \simeq \mathcal T_{\bP^1}(+1) \simeq O_{\bP^1}(+3)$  \cite[Lemma~3]{Iarrobino}.

\def\End{\mathrm{End}}
\begin{theorem}\label{T:vector} $\Mqn \to \bP^{q-1}$ is not a vector bundle in the algebraic category. 
\end{theorem}

\begin{proof} Let $\varpi_n$ denote the fibration $\Mqn \to \bP^{q-1}$. It is induced by the natural projection $A_{q,n} = k[x_1,\dots,x_q]/\mathfrak m^n \to k[x_1,\dots,x_q]/\mathfrak m^2 = A_{q,2}$. Consider the set map
\[
\mathcal M_{2,n} \to \mathcal M_{q,n}
\]
sending a commuting pair $(N_1, N_2)$ to the $q$-tuple $(N_1,N_2,\dots,N_2)$. It can be readily seen that this map is algebraic: Recall the universal sheaf over $\mathcal M_{2,n}$ from Proposition~\ref{P:universal}. We denote it by $ \mathcal U$. Consider the associated ideal sheaf 
\[
\mathcal U' := \mathcal U + \cO_{\mathcal M_{2,n}}\langle x_3-x_2, \dots, x_q - x_2 \rangle \subset \cO_{\mathcal M_{2,n}}\otimes_k A_{q,n}
\]
Note that this annihilates regular representations of $A_{q,n}$: For any $[I] \in \mathcal M_{2,n}$, we have 
\[
\mathcal U' \otimes \kappa([I]) \simeq \mathcal U\otimes \kappa([I]) + \langle x_3-x_2, \dots, x_q - x_2 \rangle
  = I + \langle x_3-x_2, \dots, x_q - x_2 \rangle \subset k[x_1, \dots, x_q]
  \]
which annihilates a regular representation since $I \subset k[x_1,x_2]$ satisfies $I \not\subset \langle x,y\rangle^2$. 

By the universality (Proposition~\ref{P:universal}), $\mathcal U'$ induces a morphism $\psi: \mathcal M_{2,n} \to \Mqn$. Let $I \subset A_{2,n}$ be a co-length $n$ ideal annihilating a regular representation, and $(N_1,N_2)$ denote the commuting pair of nilpotents associated with $I$ i.e. $I$ is the kernel of $A_{2,n} \to \End(k^n)$ given by mapping $x_i$ to $N_i$. Obviously $I + \langle x_3 - x_2, \dots, x_q - x_2 \rangle$ is contained in the kernel of $A_{q,n} \to \End(k^n)$ given by $x_1 \mapsto N_1$ and $x_i \mapsto N_2$, $\forall i \ge 2$. Then it must equal the kernel since they are of the same co-length. It follows that  $\psi$ corresponds to the map $(N_1, N_2) \mapsto (N_1, N_2, N_2, \dots, N_2)$.

Suppose that $\varpi_n$ is a vector bundle. From the commutative square
\[
\xymatrix{
\varpi_n^{-1}(\iota(\bP^1)) = \mathcal M_{2,n} \ar[r]\ar[d] & \Mqn \ar[d]^-{\varpi_n} \\
\mathcal M_{2,2}\simeq \bP^1 \ar[r]_-{\iota} & \mathcal M_{q,2} \simeq \bP^{q-1}
}
\]
we conclude that $\mathcal M_{2,n} \to \bP^1$ is also a vector bundle. But then, $\mathcal M_{2,n}$ should be isomorphic to $\oplus_{j=1}^{n-1} T_{\bP^1}(+j) \simeq \oplus_{j=1}^{n-1} \cO_{\bP^1}(j+2)$ by Theorem~\ref{T:main}. In this case, the projection $\pi_{n,3}: \Mqn \to \mathcal M_{q,3}$ induced by $A_{q,n} \to A_{q,3}$ is simply the bundle projection $\oplus_{j=1}^{n-1} \cO_{\bP^1}(j+2) \to  \cO_{\bP^1}(+3)$, hence admits a section $\sigma$. Now, consider the following diagram:
\[
\xymatrix{
\mathcal M_{2,n} \ar[r]^-{\pi_{n,4}} \ar[dr]^-{\pi_{n,3}} & \mathcal M_{2,4} \ar[d]^-{\pi_{4,3}} \\
& \mathcal M_{2,3} \ar@/^2pc/[ul]^-{\sigma}
}
\]
which is commutative since $\pi_{n,3} = \pi_{4,3}\circ \pi_{n,4}$. Hence we have a section $\pi_{n,4}\circ \sigma$ of $\pi_{4,3}$, contradicting \cite[Lemma~3]{Iarrobino}. 

\end{proof}

\bibliographystyle{alpha}
\bibliography{nilpotents}

\end{document}